%% file: aclBexp.tex
\documentclass[12pt,oneside]{amsart}
\usepackage{amssymb,amsmath,amsthm}
\usepackage[colorlinks,pdfusetitle]{hyperref}
\usepackage{enumerate}
\hypersetup{
    colorlinks=true,
    linkcolor=blue,
    filecolor=magenta,      
    urlcolor=cyan,
    citecolor=magenta
}
\usepackage{enumitem}

\usepackage{diagrams}

\usepackage[T1]{fontenc}
\usepackage{todonotes}

\usepackage{bm}
\usepackage{xfrac}
\usepackage{geometry}

\input{macros}

\theoremstyle{plain}
\newtheorem{theorem}{Theorem}[section]
\newtheorem{lemma}[theorem]{Lemma}
\newtheorem{prop}[theorem]{Proposition}
\newtheorem{proposition}[theorem]{Proposition}

\newtheorem{fact}[theorem]{Fact}
\newtheorem{corollary}[theorem]{Corollary}
\newtheorem{claim}[theorem]{Claim}

\theoremstyle{definition}
\newtheorem{definition}[theorem]{Definition}

\newtheorem*{claim*}{Claim}

\newcommand{\be}{\begin{equation}}
\newcommand{\ee}{\end{equation}}

\newcommand{\xbar}{{\ensuremath{\bar{x}}}}

\newcommand{\ybar}{{\ensuremath{\bar{y}}}}

\newcommand{\gambar}{{\ensuremath{\bar{\gamma}}}}

\newcommand{\EA}{\mathrm{EA}}
\newcommand{\ELA}{\mathrm{ELA}}
\newcommand{\B}{\mathbb{B}}

\newcommand{\Loc}{\mathrm{Loc}}

\DeclareMathOperator{\Span}{span}
\DeclareMathOperator{\bcl}{bcl}

\renewcommand{\leq}{\leqslant}
\renewcommand{\geq}{\geqslant}

\newenvironment{claimproof}{\bgroup\begin{proof}}{\end{proof}\egroup}

\def\Ind#1#2{#1\setbox0=\hbox{$#1x$}\kern\wd0\hbox to 0pt{\hss$#1\mid$\hss}
\lower.9\ht0\hbox to 0pt{\hss$#1\smile$\hss}\kern\wd0}
\def\ind{\mathop{\mathpalette\Ind{}}}
\def\Notind#1#2{#1\setbox0=\hbox{$#1x$}\kern\wd0\hbox to 0pt{\mathchardef
\nn="3236\hss$#1\nn$\kern1.4\wd0\hss}\hbox to 0pt{\hss$#1\mid$\hss}\lower.9\ht0
\hbox to 0pt{\hss$#1\smile$\hss}\kern\wd0}

\title[Algebraic types in exponential fields]{Algebraic types in Zilber's exponential field}

\author{Vahagn Aslanyan}
\address{Department of Mathematics, University of Manchester, Manchester, UK}
\email{Vahagn.Aslanyan@manchester.ac.uk}

\author{Jonathan Kirby}

\address{School of Mathematics, University of East Anglia, Norwich, UK}
\email{Jonathan.Kirby@uea.ac.uk}

\thanks{Most of this work was done while the first author was a senior research associate at the University of East Anglia, where both authors were supported by EPSRC grant EP/S017313/1. The first author has continued to work on this paper at the University of Leeds (supported by Leverhulme Trust Early Career Fellowship ECF-2022-082) and at the University of Manchester (supported by EPSRC Open Fellowship EP/X009823/1). For the purpose of open access, the authors have applied a Creative Commons Attribution (CC BY) licence to any Author Accepted Manuscript version arising from this submission.}

\date{}


\begin{document}

\begin{abstract}
    We characterise the model-theoretic algebraic closure in Zilber's exponential field. A key step involves showing that certain algebraic varieties have finite intersections with certain finite-rank subgroups of the graph of exponentiation. Mordell-Lang for algebraic tori (a theorem of Laurent) plays a central role in our proof. 
\end{abstract}

\maketitle

\section{Introduction}

Given a complex number $a$ and a subfield $K$ of the complex field, we can ask if $a$ is \emph{algebraic} over $K$, that is, satisfies a non-trivial polynomial equation with coefficients in $K$. If we look purely at the field structure then this is essentially all we can say about $a$ given $K$; in model-theoretic terms, the minimal polynomial of $a$ over $K$ (or the fact that $a$ is transcendental over $K$) completely determines the \emph{type} of $a$ over $K$.
When $a$ is algebraic over $K$, this type has only finitely many realisations: the conjugates of $a$ over $K$.

If we now add the exponential function to the field then this single notion of algebraicity generalises into a number of distinct concepts. 
One could ask for $a$ to satisfy some exponential polynomial equation with coefficients from $K$, but this is not particularly useful, as it neither pins down $a$ to a finite set, nor does this concept even give a closure operator.

The generalisation of algebraicity central to this paper is the one called \emph{model-theoretic algebraicity}: an element $a$ of a structure $M$ is algebraic over a subset $B \subs M$ if it lies in a finite set which is definable by a first-order formula with parameters from $B$. Equivalently, the type of $a$ over $B$ in $M$ has only finitely many realisations, so is said to be an algebraic type. The \emph{(model-theoretic) algebraic closure} $\acl_M(B)$ of $B$ in $M$ is the union of all the finite subsets of $M$ which are definable over $B$. We consider this concept in the case of exponential fields.

An \emph{exponential field} or \emph{E-field} is a field $F$ of characteristic 0 equipped with a homomorphism $\exp$ from its additive to its multiplicative group. We often write $e^z$ instead of $\exp(z)$. If $F$ is algebraically closed as a field, we call it an \emph{EA-field}. If, in addition, the exponential map is surjective (every non-zero element has a logarithm), we call it an \emph{ELA-field}. 
Given any exponential field $F$ and a subset $B \subseteq F$, we have the  exponential subfield $\gen{B}_F^{\mathrm{E}}$ generated by $B$, and also write $\gen{B}_F^{\mathrm{EA}}$ for the smallest  exponential subfield containing $B$ which is relatively algebraically closed in $F$ in the field-theoretic sense. We also write $\gen{B}_F^{\mathrm{ELA}}$ for the smallest E-subfield $K$ containing B which is relatively algebraically closed in $F$ and also closed under logarithms, that is, if $\exp(b) \in K$ then $b \in K$.

In any exponential field $F$, it is immediate that $\acl_F(B) \supseteq \gen{B}_F^{\mathrm{EA}}$.

The main classical examples of exponential fields are $\Rexp$ and $\Cexp$, the real and complex fields equipped with the usual exponentials.
In the celebrated paper \cite{Wilkie96}, Wilkie showed that $\Rexp$ is o-minimal. Every o-minimal field has definable Skolem functions, and so for any $B \subseteq \R$ the algebraic closure $\acl(B)$ is equal to the definable closure $\dcl(B)$ (the union of all singleton sets definable over $B$), and is the smallest elementary substructure of $\Rexp$ containing $B$.

A different generalisation of algebraic closure which is also important for exponential fields is that known as \emph{exponential algebraic closure}, denoted by $\ecl$. For any exponential field $F$ and subset $B$, $\ecl_F(B)$ consists of the components of zeros of functions defined implicitly from exponentials and polynomials (see \cite[Definition 3.2]{Kirby-Schanuel}). In $\Rexp$, $\acl$ coincides with this $\ecl$ (see \cite[Theorem~4.2]{Jones-Wilkie}) but this is not true for most exponential fields.

In the complex exponential field, the characterisation of $\acl$ is, like many model-theoretic questions, completely open. We discuss this further at the end of the paper. One observation we can make is that $\pi$ is in $\acl(\emptyset)$. 
To see this, we observe that the kernel $2\pi i\Z$ of the exponential map is defined by 
\[x \in \ker \text{ iff } e^x=1\]
and the integers $\Z$ are definable as the multiplicative stabiliser of the kernel, that is, by 
\[x \in \Z \text{ iff } \forall y[y \in \ker \to xy \in \ker].\]
The set $\{\pm 2\pi i\}$ is definable as the set of generators of the kernel, via
\[x \in \ker \wedge (\forall y \in \ker)(\exists z \in \Z)[y = zx]\]
and the result follows since $\pi$ is of course field-theoretically algebraic over $2\pi i$.
In \cite{Kirby_Macintyre_Onshuus} this argument is extended to show that $\pi$ is actually pointwise definable. We give the extended argument at the end of this paper.

\medskip

Following his philosophy that $\Cexp$ should be as model-theoretically tame as possible given that the ring $\Z$ is definable, Zilber \cite{Zilb-pseudoexp} introduced a new exponential field $\B$, satisfying many known and desirable properties of $\Cexp$, and conjectured that it is in fact isomorphic to $\Cexp$. In particular, the kernel of the exponential map in $\B$ is of the form $\tau \Z$, with $\tau$ transcendental, and an exponential field of this form is said to have \emph{standard kernel}. Exactly the same argument as above shows that in any exponential field with standard kernel we have $\tau \in \acl(\emptyset)$.

The exponential field $\B$ is characterised by three further properties: the Schanuel Property, Strong Exponential-Algebraic Closedness, and the Countable Closure Property, all of which we explain in the next section. The Schanuel property is the conclusion of Schanuel's conjecture of transcendental number theory. It can be formulated as saying that for any finite tuple $\abar$ from an exponential field $F$, a certain predimension function $\delta(\abar)$ is non-negative. We can relativise this predimension function over a subset $B$ of $F$, and we say that $B$ is \emph{strong} in $F$ if for all finite tuples $\abar$ from $F$, the predimension $\delta(\abar/B)$ is non-negative, and $B$ also contains the kernel of $F$.
It turns out that for any $B \subseteq F$, there is a smallest strong subset of $F$ containing it, called the \emph{hull} of $B$, denoted $\hull{B}_F$. We write $\hull{B}_F^{\mathrm{EA}}$ for $\gen{\hull{B}_F}_F^{\mathrm{EA}}$. In the case $F=\B$, we will drop the subscript $F$ from the above notation and just write $\hull{B}$, $\gen{B}^{\mathrm{E}}$, $\gen{B}^{\mathrm{EA}}$, and $\hull{B}^{\mathrm{EA}}$.
Our main theorem characterises the model-theoretic algebraic closure in $\B$ as follows:
\begin{theorem}\label{thm:main-thm-acl=hull-EA}
For every subset $B \subs \B$ we have
$\acl(B) = \hull{B}^{\mathrm{EA}}$.
\end{theorem}

We view this theorem as the first part of a broader project aiming to characterise all the finite rank types in $\B$. The algebraic types correspond to the rank-0 types.

\subsection{Outline of the paper}
In Section \ref{sec: prelim} we explain the other properties defining $\B$ and the hull, and we set up our notation. In Section~\ref{sec: isolation of types} we characterise the types of elements of the hull of a finite tuple over that tuple and show that they are isolated. This plays a crucial role in the rest of the paper and relies on several geometric properties of $\B$.

In order to prove Theorem \ref{thm:main-thm-acl=hull-EA}, we show first that $\acl(B) \supseteq \hull{B}^{\mathrm{EA}}$. This reduces to showing that $\acl(B) \supseteq \hull{B}$, which is done in Section~\ref{sec: hull subset acl}, and comprises the bulk of the work of the paper. This statement can be seen as a ``Diophantine'' statement about some algebraic varieties containing only finitely many points from a certain finite rank subgroup of the graph of the exponential map. This resembles the Mordell-Lang conjecture for tori, a particular case of which states that an algebraic curve living inside an algebraic torus contains only finitely many points from a finite rank subgroup of the torus unless the curve is a coset of a subtorus. Hence it should not come as a surprise that Mordell-Lang plays a crucial role in the proof of the theorem. 

We then show that $\acl(B) \seq \hull{B}^{\mathrm{EA}}$ in Section~\ref{sec: acl subset hull}.
 This inclusion follows from the Hrushovski style amalgamation-with-predimension construction of a countable submodel of $\B$, combined with the first inclusion. The reader familiar with Hrushovski constructions may be aware that in structures obtained by such a construction in certain settings (e.g. Hrushovski's ab initio example) the model-theoretic algebraic closure can be described in terms of the hull. In our setting this yields only a weak version of the second inclusion; some non-trivial work and the first inclusion are required to prove the full case.

In the final Section~\ref{sec: final comments} we compare the algebraic closure with other closure operators in $\B$, including the definable closure and various bounded closure operators, discuss what we can say about algebraic closure in $\Cexp$, and leave some open questions.

\section{Zilber's exponential fields}\label{sec: prelim}

Zilber's exponential field $\B$ is the unique (up to isomorphism) ELA-field of cardinality continuum with standard kernel which satisfies three further properties: the Schanuel Property, Strong Exponential-Algebraic Closedness, and the Countable Closure Property. We explain these properties and the relevant surrounding concepts, roughly following the presentation in \cite{Kirby-pseudo}.

\subsection{The Schanuel Property, predimension, and strong hull}

We begin by introducing the notion of \textit{relative predimension} of a $\Q$-vector space over another. We work in an ambient exponential field $F$. 

Given two $\Q$-vector subspaces $L_1 \seq L_2$ of $F$ such that the quotient $L_2/L_1$ is finite dimensional, we define the \textit{relative predimension} of $L_2$ over $L_1$ as
 \[ \delta(L_2/L_1) := \td (L_2, \exp(L_2) / L_1, \exp(L_1)) - \ldim_{\Q} (L_2/L_1) \]
 where $\td$ and $\ldim$ stand for transcendence degree and linear dimension respectively.
 
 For an arbitrary $\Q$-vector subspace $L_1$ of $F$ and a finite dimensional $\Q$-vector subspace $L$ of $F$ we define
 \[ \delta(L/L_1) := \delta(\Span_{\Q}(L_1\cup L) / L_1). \]

 For two finite tuples $\abar, \bbar$ we let $\delta(\abar/\bbar) := \delta(\Span_{\Q}(\abar) / \Span_{\Q}(\bbar))$.

Given $\Q$-vector subspaces $L_1\seq L_2$ of an exponential field $F$, we say $L_1$ is \textit{strong} in $L_2$ and write $L_1 \strong L_2$ if $\ker(\exp) \cap L_2 \subseteq L_1$ and for all finite dimensional $\Q$-linear subspaces $L\seq L_2$ we have $\delta(L/L_1)\geq 0$. 

For a tuple $\bbar$ from a $\Q$-linear subspace $L$ we say $\bbar$ is strong in $L$, and write $\bbar \strong L$, if $\Span_{\Q}(\bbar) \strong L$.

For a subset $B\seq F$ we define the \emph{strong hull} or just \emph{hull} of $B$ as follows:
    \[ \hull{B}_F := \bigcap \{ L \seq F: L \mbox{ is a $\Q$-vector subspace}, B \cup \ker(\exp) \seq L \mbox{ and } L \lhd F \}. \] 

One can show that the intersection of strong $\Q$-subspaces is again strong, and hence $\hull{B}_F$ is the smallest strong subspace of $F$ containing $B$. As mentioned in the introduction, we write $\hull{B}_F^{\EA}$ for  $\gen{\hull{B}}_F^{\EA}$, the smallest EA-subfield of $F$ which contains $B$ and is strong in $F$.

For a finite-dimensional $\Q$-vector space $L$ we define its \textit{predimension} as \[ \delta(L):=\delta(L/\Span_{\Q}(\ker(\exp)))\]
and likewise for a finite tuple $\abar$ we define $\delta(\abar) \leteq \delta(\Span_\Q(\abar))$.

An exponential field $F$ satisfies the \emph{Schanuel Property} if the conclusion of Schanuel's conjecture holds in $F$, that is, for any $z_1,\ldots,z_n\in F$
  \[\td_{\Q}(z_1,\ldots,z_n, e^{z_1}, \ldots, e^{z_n}) - \ldim_{\Q}(z_1,\ldots,z_n) \geq 0 .\]
For exponential fields $F$ with standard kernel (or trivial kernel), it is easy to see that $F$ satisfies the Schanuel Property if and only if for every finite-dimensional $\Q$-vector subspace $L\seq F$ we have $\delta(L)\geq 0$.
In this case, it easily follows that if $\ldim_\Q B$ is finite then $\ldim_{\Q} \hull{B}_F$ is finite and $\delta(\hull{B}_F)$ is the smallest possible predimension of an extension {(as a finite dimensional $\Q$-vector space)} of $B$ in $F$.

\subsection{Strong Exponential-Algebraic Closedness and the Countable Closure Property} \ 

For an $l \times n$ matrix $M$ of integers we define $M:\ga^n \cross \gm^n \to \ga^l \cross \gm^l$ to be the map given by $M:(\bar{x},\bar{y}) \mapsto (M\xbar, \ybar^M)$ where $(M\xbar)_i = \sum_{j=1}^n m_{ij}x_j \mbox{ and } (\ybar^M)_i = \prod_{j=1}^n y_j^{m_{ij}}.$

 An irreducible subvariety $V \subseteq \ga^n \cross \gm^n$ is \emph{rotund} if for any $1 \leq l \leq n$ and any $l\times n$ matrix $M$ of integers $\dim M(V) \geq \rk M$.

 A subvariety $V$ as above is called \textit{free} if for any $M$ as above with $l=\rk(M) =1$, both the additive and multiplicative projections of $M(V)$ have dimension 1.

Given an exponential field $F$ and a subvariety 
 $V \subs \ga^n \cross \gm^n$, we define
\[V^\dagger_{\bbar} \leteq \{\abar :(\abar,e^{\abar}) \in V \text{and the tuple $\abar$ is $\Q$-linearly independent over $\bbar$}\}.\]

When $F$ has standard kernel, $\Q$ is $\emptyset$-definable in $F$, hence if $V$ is defined (in the field-theoretic sense) over $\gen{\bbar}^\mathrm{E}_F$ then $V^\dagger_{\bbar}$ is a $\bbar$-definable set.

We say an exponential field $F$ is
\emph{Strongly Exponentially-Algebraically Closed} if it is an ELA field and for every \emph{free} and \emph{rotund} subvariety $V\seq \ga^n(F) \cross \gm^n(F)$ of dimension $n$ {defined over $F$} and for every finite tuple $\bbar$ {from $F$}, the set $V_{\bbar}^{\dagger}$ is non-empty.

An exponential field $F$ is said to have the \emph{Countable Closure Property} if for every countable subset $B$, the exponential-algebraic closure $\ecl_F(B)$ is countable. We have not carefully defined $\ecl$, and in fact we will not make any use of this property (save for some comments in section~\ref{sec: final comments}) but for those exponential fields with standard kernel which satisfy the Schanuel Property there is a simpler equivalent formulation:\\ 
For every free and rotund subvariety $V\seq \ga^n(F) \cross \gm^n(F)$ of dimension $n$ and for every finite tuple $\bbar$, the set $V_{\bbar}^{\dagger}$ is countable.

For each uncountable cardinal $\kappa$ there is a unique (up to isomorphism) ELA-field of cardinality $\kappa$ with standard kernel satisfying the Schanuel Property, Strong Exponential-Algebraic Closedness, and the Countable Closure Property \cite{Zilb-pseudoexp,Bays-Kirby-exp}. The one of cardinality $2^{\aleph_0}$ is denoted by $\B$ as already pointed out above.

\section{Isolation of types}\label{sec: isolation of types}

To understand when $a \in \acl(\bbar)$ in $\B$, we need to know something about the types of finite tuples which are realised in $\B$. We are interested in syntactic complete types, that is, $\tp(a/\bbar) = \class{\phi(x,\bbar)}{\B \models \phi(a,\bbar)}$ where $\phi(x,\ybar)$ runs over first-order formulas. The construction of $\B$ is done in a semantic way, so the natural notion of type which arises is a Galois type, that is, an automorphism orbit. As $\B$ is not a saturated, or even homogeneous, model of its first-order theory, Galois types give in general a slightly finer equivalence relation than syntactic types. However, $\B$ is quasiminimal excellent and hence $\aleph_0$-homogeneous over the empty set \cite[Definition 1.1]{Kirby-OQMEC} and it follows that for those types of finite tuples (over a finite tuple of parameters) which are realised in $\B$, the two do coincide. Thus we have the following statement.

\begin{fact}
    If $\abar, \abar' \in \B^n$ and $\bbar \in \B^k$ then $\tp(\abar/\bbar) = \tp(\abar'/\bbar)$ if and only if there is an automorphism of $\B$ fixing $\bbar$ pointwise and mapping $\abar$ to $\abar'$.
\end{fact}

Since $\tau \in \acl(\emptyset)$, it is convenient to work over tuples containing $\tau$, so from now on we focus on types over such tuples. 

\begin{definition}
   {Let $F$ be an algebraically closed field.  For a subfield $K$ of $F$ and a tuple $\bar{u}$ from  $F$ the \emph{locus} of $\bar{u}$ over $K$, denoted by $\loc(\bar{u}/K)$, is the Zariski closure of $\bar{u}$ over $K$.}
\end{definition}

\begin{prop}\label{prop: isolating types over strong sets}
Let $\bbar$ be a finite tuple from $\B$ containing $\tau$, such that $\bbar \strong \B$.
Suppose $\abar$ is a finite tuple from $\B$ which is $\Q$-linearly independent over $\bbar$ and that $\delta(\abar/\bbar) = 0$.

Then there is a positive integer $m$ such that, after replacing $\abar$ by $\abar/m$, we get that
$\tp(\abar/\bbar)$ is isolated by the formula $\xbar \in V^\dagger_\bbar$, where $V = \loc(\abar,\exp(\abar)/B)$ with $B:=\Q\left( \Span_{\Q}(\bbar), \exp(\Span_{\Q}(\bbar))\right)$.
\end{prop}

\begin{proof}
We give a proof referring to the paper \cite{Bays-Kirby-exp}, which however uses a somewhat different and more general setup, of what are there called $\Gamma$-fields. In our situation, a $\Gamma$-field is a subfield $F$ of $\B$ with a distinguished $\Q$-linear subspace $D = D(F)$ containing $\tau$, and the restriction of $\exp$ to $D$. Furthermore, $F$ should be generated as a subfield of $\B$ by $D \cup \exp(D)$. It is finitely generated in the sense of that paper if and only if $D$ is finite-dimensional as a $\Q$-vector space. In \cite{Kirby-FPEF} and elsewhere, this concept is called a partial exponential field.

With these definitions and notation, $B$ is the $\Gamma$-field with $D(B)$ spanned by $\bbar$. Take $A$ to be its extension with $D(A)$ spanned by $\abar,\bbar$. Then \cite[Proposition~3.22]{Bays-Kirby-exp} says that, after replacing $\abar$ by $\abar/m$ for a suitable $m$, we get that $\abar$ is a \textit{good basis}, which by \cite[Definition~3.19]{Bays-Kirby-exp} means that the extension $A$ is determined up to isomorphism by the locus $V$ of $(\abar,\exp(\abar))$ over $B$ for a basis $\abar$, that is, $V$ together with the information that the basis $\abar$ is $\Q$-linearly independent over $\bbar$.

Now we note that since $\bbar \strong \B$, the extension $B \subseteq A$ is necessarily strong. Further, since $\delta(\abar/\bbar) = 0$, we conclude that $A \strong \B$.

The exponential field $\B$ is constructed as an elementary extension of a countable model $\mathbb F$ \cite[Lemma 5.4]{Kirby-OQMEC}, and this $\mathbb F$ is the Fra\"iss\'e limit of the category of finitely generated $\Gamma$-fields with strong extensions \cite[Theorem~5.9]{Bays-Kirby-exp}. So from the standard homogeneity of Fra\"iss\'e limits, we see that if $A$ and $A'$ are isomorphic strong finitely generated extensions of $B$ within $\mathbb F$ then the isomorphism extends to an automorphism of $\mathbb F$, and hence the corresponding generators of $A$ and $A'$ over $B$ have the same Galois type over $B$, and then in particular they have the same syntactic type over $\bbar$. The proposition follows.
\end{proof}

Now consider the situation where $\bbar$ is a finite tuple from $\B$ which is not strong in $\B$, and $\abar$ is a {$\Q$-linear} basis for $\hull{\bbar}$ over $\bbar$. 
We can assume that $\bbar$ contains $\tau$ and is $\Q$-linearly independent. Then $\abar\bbar$ is a $\Q$-linear basis for $\hull{\bbar}$, and replacing $\abar$ by $\abar/m$ and $\bbar$ by $\bbar/m$ for a suitable $m \in \N$, we may assume it is a good basis.

\begin{prop}\label{prop: isolating types over any sets}
Let $\bbar$ be a finite tuple from $\B$ containing $\tau$ and suppose $\abar\bbar$ is a good basis of $\hull{\bbar}$. Then $\tp(\abar/\bbar)$ is isolated by the formula $\xbar \in V^\dagger_\bbar$, where $V = \loc(\abar,\exp(\abar)/B)$ with $B:=\Q\left( \Span_{\Q}(\bbar), \exp(\Span_{\Q}(\bbar))\right)$. Furthermore, all realisations $\abar'$ of this type are such that the $\Q$-linear span of $\abar'\bbar$ is equal to that of $\abar\bbar$, namely $\hull{\bbar}$.
\end{prop}

\begin{proof}
If $\abar$ is non-empty then $\delta(\abar/\bbar)<0$, so we are not in the situation of Proposition~\ref{prop: isolating types over strong sets}. To get to that situation, we take a subtuple $\bbar'$ of $\bbar$ which contains $\tau$ and is otherwise exponentially algebraically independent in $\B$, of maximal length, say $d+1$, and let $\cbar = \bbar \setminus \bbar'$, the rest of the tuple. Then $\bbar' \strong \B$ with $\delta(\bbar') = \delta(\bbar'\cbar\abar) = d$, and $\delta(\cbar\abar/\bbar') = 0$.

Let $U := \loc(\cbar\abar,\exp(\cbar\abar)/\bbar',\exp(\bbar'))$. Then by Proposition~\ref{prop: isolating types over strong sets}, $\tp(\cbar\abar/\bbar')$ is isolated by the formula $\ybar\xbar \in U^\dagger_{\bbar'}$. Evaluating at $\ybar=\cbar$ we see that $\tp(\abar/\bbar'\cbar)$ is isolated by $\xbar \in V^\dagger_{\bbar'\cbar}$, that is, $\xbar \in V^\dagger_\bbar$, as required.

Now suppose that $\abar' \in V^\dagger_\bbar$. Then $\tp(\abar'/\bbar) = \tp(\abar/\bbar)$ so $V = \loc(\abar',\exp(\abar')/\bbar,\exp(\bbar))$, and $\ldim_\Q(\abar'/\bbar) = \ldim_\Q(\abar/\bbar)$. Thus $\delta(\abar'/\bbar) = \delta(\abar/\bbar)$, and by the minimality of this predimension and the uniqueness of the hull, $\abar'$ must also be a basis for $\hull{\bbar}$ over $\bbar$. 
\end{proof}

From the $\aleph_0$-homogeneity of $\B$ over $\emptyset$ as mentioned above, we get:
\begin{corollary}\label{cor: locus -> automorphism}
    Let $\abar, \bbar$, and $B$ be as in Proposition~\ref{prop: isolating types over any sets}. Let $\abar'$ be another tuple such that $ \loc(\abar,\exp(\abar)/B) = \loc(\abar',\exp(\abar')/B).$
    Then there is an automorphism of $\B$ fixing $B$ pointwise and mapping $\abar$ to $\abar'$. \qed
\end{corollary}

\section{The strong hull is contained in the algebraic closure}\label{sec: hull subset acl}

The main technical part of the proof is done in the following proposition where we consider a locus of a special form.

\begin{proposition}\label{prop: fin many pairs (M,beta) assuming constant coord}
    Let $\bbar$ be a finite tuple from $\B$ containing $\tau$, and let $\abar \in \B^n$ be a good basis of the strong hull $\hull{\bbar}$. Let also $V \leteq \loc\left(\abar,\exp(\abar)/ B\right) \seq \ga^n(\B) \times \gm^n(\B)$ where $B:=\Q\left( \Span_{\Q}(\bbar), \exp(\Span_{\Q}(\bbar))\right)$. Assume that $W :=\pr_{\gm^n}V$ is a coset of an algebraic subgroup of $\gm^n$. Then $\abar \seq \acl(\bbar)$.
\end{proposition}

\begin{proof}
We proceed to the proof by induction on $n$, the length of $\abar$. The base case $n=0$ holds vacuously. If $n\geq 1$ then $\delta(\abar/\bbar)<0$ which means $\dim V < n$. In particular, if $n=1$ then $\dim V = 0$ and so $V$ is finite, and the conclusion of the theorem holds. 
 
So let $n\geq 2$ and assume now that for any tuple $\bbar'$ for which 
\begin{itemize}
    \item $l:=\ldim_{\Q}(\hull{\bbar'}/\bbar')<n$, and
    \item $\pr_{\gm^l}\loc\left(\abar',\exp(\abar')/ \Q(\bbar', \exp(\bbar'))^{\alg}\right)$ is a coset of a subgroup of $\gm^l$ where $\abar'\in \B^l$ is a basis of $\hull{\bbar'}$ over $\bbar$,
\end{itemize}
we have $\abar' \seq \acl(\bbar')$.

We need to show that the conclusion holds for $\abar$ and $\bbar$ as in the statement of the proposition. Let $x_1,\ldots,x_n,y_1,\ldots,y_n$ denote the coordinates on $\ga^n\times \gm^n$. Observe that by a simple change of coordinates we may reduce to the case where $W$ is defined by equations of the form $y_k = c_k,~ k=1,\ldots,s$ for some $c_k \in  B$ and some $s\leq n$. 

\begin{claim}\label{claim: one alg element}
    There is $a'\in \Span (\abar, \bbar) \setminus \Span (\bbar)$ such that $a'\in \acl(\bbar) $. 
\end{claim}

We now explain why the claim is enough to prove the induction step. If the claim holds, then we may assume $a'$ is one of the $a_i$, say $a_{t}$.  Then the rest of the $a_i$, denoted $\abar_*$ as a tuple, form a basis of $\hull{\bbar,a'}$ over $\bbar,a'$. Let $B':=B(a', \exp(\Q a'))$. Since by our assumption $e^{a_1},\ldots,e^{a_s}\in B$ and $e^{a_{s+1}},\ldots,e^{a_n}$ are algebraically independent over $B$, we see that the multiplicative projection of $\loc\left(\abar,\exp(\abar))/ B'\right)$ is defined by the equations $y_k=c_k,~ k=1,\ldots,s$, and $y_t=e^{a'}$ and is a coset of a subtorus of $\gm^n$. Therefore, the multiplicative projection of $\loc\left(\abar_*,\exp(\abar_*))/ B'\right)$ is a coset of a subtorus of $\gm^{n-1}$. Also, $\dim (\hull{\bbar,a'} / \bbar,a') = n-1$, hence by the induction hypothesis we have $\hull{\bbar} = \hull{a', \bbar} \seq \acl(a',\bbar)= \acl(\bbar)$.

Now we prove the claim considering two cases.

\subsection*{Case 1.} $\dim W < n - 1$. This means $W$ has at least two constant coordinates, $y_1$ and $y_2$.
Let $U \seq \B^2$ be the projection of $V_{\bbar}^{\dagger}$ to the first two coordinates. Clearly $U \seq (a_1+\tau \Z)\times (a_2+\tau \Z)$. Suppose for some $m_1,m_2,l_1,l_2\in \Z$ we have $(a_1+m_1\tau, a_2+m_2\tau), (a_1+l_1\tau, a_2+l_2\tau) \in U$. We may assume $m_1\geq 0,~ l_1\geq 0$. By Corollary~\ref{cor: locus -> automorphism}, the maps
\[ \mu : (a_1, a_2) \mapsto (a_1+m_1\tau, a_2+m_2\tau),~  \lambda : (a_1, a_2) \mapsto (a_1+l_1\tau, a_2+l_2\tau)\]
extend to automorphisms over $B$. 

Now, the automorphism $\mu^{l_1}\lambda^{-m_1}$ maps $(a_1,a_2)$ to $(a_1, a_2 + (l_1m_2-l_2m_1)\tau)$. If $l_1m_2-l_2m_1 \neq 0$ then the orbit of $a_2$ is infinite under the powers of $\mu^{l_1}\lambda^{-m_1}$ while $a_1$ is fixed. However, by the induction hypotheses, $a_2$ is model-theoretically algebraic over $\bbar a_1$, so this is not possible. Therefore, we deduce that $l_1m_2-l_2m_1 = 0$.

Let $d_i := \gcd(l_i, m_i)$. Then $l_i = d_i l_i',~ m_i = d_i m_i'$ with $\gcd(l_i', m_i')=1$. Substituting this in the above equality we see that 
$ l_1'm_2' = l_2'm_1'$ which then implies that  $l_1' = l_2', m_1' = m_2'$ or $l_1' = -l_2', m_1' = -m_2'$.

Since $\gcd(l_1',m_1') = 1$, there are integers $u, v$ such that $m_1'u+l_1'v = 1$. Then $m_2'u+l_2'v = \pm 1$. Thus, we have 
\[ m_1u + l_1 v = d_1,~ m_2u + l_2v = \pm d_2. \]
Consider the automorphism $\sigma := \mu^u \lambda^v : (a_1,a_2) \mapsto (a_1+d_1\tau, a_2 \pm d_2\tau)$. Then on $(a_1,a_2)$ the automorphism $\mu$ (respectively $\lambda$) agrees with $\sigma^{m_1'}$ (respectively $\sigma^{l_1'}$), regardless of the sign in the second coordinate above. This shows that on $(a_1,a_2)$ all automorphisms over $B$ agree with powers of an automorphism  $\sigma: (a_1,a_2) \mapsto (a_1+d_1\tau, a_2 + d_2\tau)$ where $|d_1|,|d_2|$ are minimal. Since $\sigma^r$ maps $(a_1,a_2)$ to $(a_1+rd_1\tau, a_2 + rd_2\tau)$, we conclude that any point in $U$ must be of the latter form for some integer $r$. Then the formula
\[ \varphi(x):= \exists x_1,\ldots,x_n \left((x_1,\ldots,x_n)\in V_{\bbar}^{\dagger} \wedge x = d_1x_2 - d_2x_1 \right)\]
defines $a':=d_1a_2 - d_2a_1$ over $\bbar$. Since $\abar$ is $\Q$-linearly independent over $\bbar$, if $a'\in \Span(\bbar)$, then $d_1=d_2=0$, and $U$ is a singleton and so $a_1,a_2\in \acl(\bbar)$. Otherwise  $a'\in \Span (\abar, \bbar) \setminus \Span (\bbar)$ and $a'\in \acl(\bbar) $ which proves Claim~\ref{claim: one alg element}.

\subsection*{Case 2.} $\dim W = \dim V = n-1$. In this case $V$ projects dominantly to the coordinates $\ybar_* :=( y_2,\ldots,y_n)$, the coordinate $y_1$ is constant (equal to $c_1$), and every coordinate $x_k$ is algebraic over $\ybar_*$ on $V$.
Let $M \in \GL_{n-1}(\Q)$ and $\bebar_* \in \Span_{\Q}(a_1,\bbar)^{n-1}$ and $\beta_1\in \tau\Z$ be such that $(a_1+\beta_1, M\abar_*+\bebar_*) \in V_{\bbar}^{\dagger}$, where $\abar_* := (a_2,\ldots,a_n)$. Observe that an arbitrary element of $V_{\bbar}^{\dagger}$ is of the form $(a_1+\beta_1, M\abar_*+\bebar_*)$, and we aim to show that $\beta_1$ must necessarily be $0$.

Recall that $B=\Q\left( \Span_{\Q}(\bbar), \exp(\Span_{\Q}(\bbar))\right)$. Write $\cbar_*$ for $\exp(\abar_*)$. Since $\dim W = n-1$, the elements $c_2,\ldots,c_n$ are algebraically independent over $B$. Consider the minimal polynomial of $a_1$ over \[C:=B\left(c_k^{\sfrac{1}{l}}: l\in \Z^+,2\leq k \leq n\right) = B\left(\exp\left(\frac{a_k+r\tau}{l}\right): l\in \Z^+, r\in \Z,2\leq k \leq n\right).\] 
Only finitely many values of $c_k^{\sfrac{1}{l}}$ occur in that polynomial, so by a $\Q$-linear change of variables we may assume it has coefficients in $B(\cbar_*)$. 
Thus, there is an irreducible polynomial $p(X_1, \bar{Y}_*) \in B[X_1, Y_2, \ldots, Y_n]$ such that $p(X_1, \cbar_*)$ is the minimal polynomial of $a_1$ over $C$.

Since the map $(a_1,\abar_*) \mapsto (a_1+\beta_1, M\abar_*+\bebar_*)$ extends to an automorphism $\sigma$ over $B$, we have  $p(a_1+\beta_1, \bar{c}_*^M \cdot \gambar_*) = 0$ where $\gambar_* := e^{\bebar_*} \in \exp \left(\Span_{\Q}(a_1,\bbar)\right)^{n-1} \seq B^{n-1}$ for $e^{a_1} = c_1\in B$. As $M$ is a rational matrix, there is some ambiguity when we compute $\cbar_*^M$, so the last sentence means that for some choice of roots the equality holds. In fact, for any choice of roots there is a $\gambar_*$ for which the equality holds because different roots differ just by a root of unity and these factors can be included in $\gambar_*$. 

Thus, the polynomial $p\left(X_1+\beta_1, \bar{c}_*^M \cdot \gambar_*\right)$ vanishes at $a_1$ and is defined over $C$. Therefore, it must be divisible by the minimal polynomial $p(X_1,\cbar_*)$. Since these two have the same degree, they must agree up to a constant factor, that is, 
\[ p\left(X_1+\beta_1, \bar{c}_*^M \cdot \gambar_*\right) = \frac{f(\cbar_*^{\sfrac{1}{t}})}{g(\cbar_*^{\sfrac{1}{t}})}\cdot p(X_1, \cbar_*)  \]
for some polynomials $f, g \in B[\bar{Y}_*]$ and some positive integer $t$. Since $\cbar_*$ is generic over $\beta_1, \gambar_*$, the following equality holds:

\[g(\bar{Y}_*^{\sfrac{1}{t}}) p\left(X_1+\beta_1, \bar{Y}_*^M \cdot \gambar_*\right) = f(\bar{Y}_*^{\sfrac{1}{t}}) p(X_1, \bar{Y}_*). \]
Here all terms are regarded as Puiseux polynomials, more precisely, Laurent polynomials in $\bar{Y}_*^{\sfrac{1}{m}}$ for a suitable $m$ (any integer that is divisible by all denominators appearing in exponents of $\bar{Y}_*$). The ring of Laurent polynomials is a unique factorisation domain (as it is a localisation of the ring of polynomials), where we may assume that $f$ and $g$ are coprime. We claim that $f(\bar{Y}_*^{\sfrac{1}{t}})$ and $g(\bar{Y}_*^{\sfrac{1}{t}})$ must be units in that ring (monomials). Indeed, $g(\bar{Y}_*^{\sfrac{1}{t}})$ must divide $p(X_1, \bar{Y}_*)$ which implies that $p(X_1, \bar{Y}_*)$ is divisible by a polynomial in $\bar{Y}_*$, namely, any irreducible factor of the product of all Galois conjugates of $g(\bar{Y}_*^{\sfrac{1}{t}})$ over $B(\bar{Y}_*)$. Since $p(X_1, \bar{Y}_*)$ is irreducible, $g(\bar{Y}_*^{\sfrac{1}{t}})$ must be a unit, i.e. a monomial in $\bar{Y}_*^{\sfrac{1}{m}}$. Similarly,  $f(\bar{Y}_*^{\sfrac{1}{t}})$ is a monomial in $\bar{Y}_*^{\sfrac{1}{m}}$.

Thus, we have
\be\label{eq:dimW=n-1_poly_equality}
p\left(X_1+\beta_1, \bar{Y}_*^M \cdot \gambar_*\right) = \xi \cdot \bar{Y}_*^{\bar{u}} \cdot p(X_1, \bar{Y}_*) \mbox{ for some } \xi \in B,~ \bar{u}\in \Q^{n-1}.
\ee

Let $S := \{ \bar{s}\in \Z^{n-1}: \bar{Y}_*^{\bar{s}} \mbox{ occurs in } p(X_1, \bar{Y}_*) \} \seq \Z^{n-1}$ be the support of $p(X_1, \bar{Y}_*)$ in $\bar{Y}_*$ and let $h:= |S|$ (the size of $S$). If $N$ is the transpose of $M$ then equality \eqref{eq:dimW=n-1_poly_equality} implies that $N(S) = S + \bar{u}$ (equality as sets), that is, $N$ permutes $S$ and translates by $\bar{u}$. So there is a permutation $\mu$ on $h$ elements such that 
\[ N\bar{s} = \mu(\bar{s}) + \bar{u} \mbox{ for all } \bar{s}\in S. \]

So for every $r\in \Z$ we have 
\[ N^r \bar{s} = \mu^r(\bar{s}) + (I+\ldots+N^{r-1}) \bar{u} \mbox{ for all } \bar{s}\in S. \]

In particular, for $r=h!$ we have $\mu^{h!} = \id$ hence
\[ N^{h!} \bar{s} = \bar{s} + \bar{v} \mbox{ for all } \bar{s}\in S, \mbox{ where } \bar{v} = (I+\ldots+N^{h!-1}) \bar{u}. \]

Now, from \eqref{eq:dimW=n-1_poly_equality} we deduce that for all $r$
\[ p\left(X_1+r\beta_1, \bar{Y}_*^{N^r} \cdot \gambar_*^{I+\ldots+N^{r-1}}\right) = \xi^r \cdot \bar{Y}_*^{(I+\ldots+N^{r-1})\bar{u}} \cdot p(X_1, \bar{Y}_*). \] 
In particular,
\[ p\left(X_1+h!\beta_1, \bar{Y}_*^{N^{h!}} \cdot \bar{\delta} \right) = \xi^{h!} \cdot \bar{Y}_*^{\bar{v}} \cdot p(X_1, \bar{Y}_*) \mbox{ where } \bar{\delta}:=\gambar_*^{I+\ldots+N^{h!-1}}. \]

Now pick $\bar{s}\in S$ such that the coefficient of $\bar{Y}_*^{\bar{s}}$ in $p(X_1, \bar{Y}^*)$ is a non-constant polynomial $q(X_1)\in B[X_1]$. Comparing the coefficients of $\bar{Y}_*^{\bar{s}+\bar{v}}$ in both sides in the above equality we get

\[ q(X_1+h!\beta_1) \cdot \bar{\delta}^{\bar{s}} = \xi^{h!} \cdot q(X_1). \]

The comparison of the leading coefficients yields $\bar{\delta}^{\bar{s}} = \xi^{h!}$. Therefore, $q(X_1+h!\beta_1) =q(X_1)$. Since $q$ is non-constant and therefore cannot be periodic, we conclude that $\beta_1=0$.

Thus,  we showed that for any element in $V_{\bbar}^{\dagger}$ the first coordinate is $a_1$, hence $a_1 \in \acl(\bbar)$, and the claim follows. This finishes the proof.
\end{proof}

We now show that in the above proposition the condition on $W$ can be dropped by proving that the general case can always be reduced to that situation. To that end we will need the Mordell-Lang conjecture for algebraic tori:

\begin{fact}[Mordell-Lang for algebraic tori, {\cite{laurent}}]

Let $K$ be an algebraically closed field of characteristic $0$. Let also $W \subs \gm^n(K)$ be an irreducible proper algebraic subvariety, and $\Gamma \subs \gm^n(K)$ a subgroup of finite rank. 

Then there is a finite number $N \in \N$, there are $\bar{\gamma}_1,\ldots,\bar{\gamma}_N \in \Gamma$, and algebraic subgroups $H_1,\ldots,H_N$ of $\gm^n$, such that for each $i$ we have $\bar{\gamma}_i H_i \subseteq W$ and 
\[W\cap \Gamma = \bigcup_{i=1}^N \bar{\gamma}_i  H_i \cap \Gamma.\]
\end{fact}

\begin{proposition}\label{prop: acl contains hull}
Let $D \subseteq \B$ be any subset. Then $\hull{D}^{\EA}\seq \acl(D)$.
\end{proposition}
 
\begin{proof}
By finite character and the definition of the EA-closure it suffices to show that if $\bbar \seq \B$ is any finite tuple containing $\tau$ and $\abar$ is a good basis of the strong hull $\hull{\bbar}$ over $\bbar$ then $\abar \seq \acl(\bbar)$.  
    
Suppose for a contradiction there is a tuple $\bbar$ for which this is not true, and let $V \leteq \loc\left(\abar,\exp(\abar)/ B\right)$ where, as above, $B:=\Q\left( \Span_{\Q}(\bbar), \exp(\Span_{\Q}(\bbar))\right)$. Let $\bbar^*$ be a tuple from $\acl(\bbar)$ extending $\bbar$ over which $\loc(\abar,\exp(\abar)/ \acl(\bbar))$ is defined. Then $\hull{\bbar^*}\nsubseteq \acl(\bbar) = \acl(\bbar^*)$. So we replace $\bbar$ with $\bbar^*$ and assume $V=\loc(\abar,\exp(\abar)/ \acl(\bbar))$.

Let $\Gamma_0 := \exp(\hull{\bbar}) \le \gm(\B)$. Since $\hull{\bbar}$ is a finite-dimensional $\Q$-vector space, $\Gamma_0$ is a finite rank subgroup of $\gm(\B)$. It is also definable over $\bbar$, for the formula $\xbar \in V_{\bbar}^\dagger$ isolates $\tp(\abar/\bbar)$. More precisely, $\Gamma_0$ is defined by the formula \[ \left( \exists \xbar\in V_{\bbar}^{\dagger} \right) \left(\exists \bar{m}\in \Z^{|\abar|}\right) \left( \exists \bar{k}\in \Z^{|\bbar|}\right) \left[ y = \exp \left(\sum_i m_ix_i + \sum_j k_jb_j \right)\right].\] Hence $\Gamma := \Gamma_0^n \le \gm^n$, where $n:=|\abar|$, is also definable over $\bbar$. 

Let $W:= \pr_{\gm^n}V$. Clearly, we have \[ \exp(\abar) \in {\exp(V^\dagger_{\bbar})} \subs W \cap \Gamma.\] By Mordell-Lang, there are a positive integer $N$ and $\kappa_1, \ldots, \kappa_N \seq W$ such that
\begin{equation}\label{eq:ML-W-Lambda}
    W \cap \Gamma = \bigcup_{j=1}^N (\kappa_j \cap \Gamma),
\end{equation} where each $\kappa_j$ is a coset of a connected algebraic subgroup $T_j$ of $\gm^n$. We may assume $\kappa_j$'s are maximal with respect to inclusion. This also means that each set $\kappa_j\cap \Gamma$ is maximal, for it is Zariski dense in $\kappa_j$.

\begin{claim}\label{claim}
    Each $\kappa_j$ is {field-theoretically} definable over $\acl(\bbar)$.
\end{claim}
\begin{claimproof}
{
We will prove that $\kappa_1$ is definable over $\acl(\bbar)$. Suppose $\kappa_1,\ldots,\kappa_M$ are all the cosets of $T_1$ among $\kappa_1, \ldots, \kappa_N$, i.e. $T_1=\ldots=T_M\neq T_j$ for all $j>M$.    Consider the set 
   \[  \Gamma_1:=\left\{ \gambar \in \Gamma: \gambar T_1 \seq W \wedge \bigwedge_{M<j\leq N} \gambar T_j \nsubseteq W \right\}.  \]
    We claim that $\Gamma_1 \subseteq \bigcup_{j\leq M} \kappa_j $. Indeed, if there exists $\gambar \in \Gamma_1 \setminus \bigcup_{j\leq M} \kappa_j $ then $\gambar T_1 \seq W$ hence $\gambar T_1\seq \kappa_j$ for some $j>M$, which then implies $\gambar T_j \seq W$, a contradiction. }
    
    {Since the $\kappa_j$ are maximal, $\Gamma_1$ contains at least one point from each $\kappa_j$ for $j\leq M$.}

   {The set $\Gamma$ is definable over $\bbar$, hence so is $\Gamma_1$. Now if $T_1$ satisfies an equation of the form $y_1^{m_1}\cdots y_n^{m_n}=1$, then $y_1^{m_1}\cdots y_n^{m_n}$ is constant on  $\kappa_j$ for all $j\leq M$. The finite set of all these constants is definable over $\bbar$ as the set of the values of the definable function $y_1^{m_1}\cdots y_n^{m_n}$ on $\Gamma_1$. Thus, each of these constants is in $\acl(\bbar)$. This means that $\kappa_1$ is field-theoretically definable over $\acl(\bbar)$.}

   A similar argument shows that every $\kappa_j$ is definable over $\acl(\bbar)$. That concludes the proof of Claim~\ref{claim}.
\end{claimproof}

To conclude the proof of Proposition~\ref{prop: acl contains hull} we observe that by \eqref{eq:ML-W-Lambda}, $\exp(\abar)\in \kappa_j$ for some $j$. But we also know that $W = \loc(\exp(\abar)/\acl(\bbar))$, and so $W \subseteq \kappa_j$. On the other hand, $\kappa_j\seq W$ hence $W = \kappa_j$.

We can now apply Proposition \ref{prop: fin many pairs (M,beta) assuming constant coord} and deduce that $\hull{\bbar} \seq \acl(\bbar)$.
\end{proof}

\section{Upper bound for the algebraic closure}\label{sec: acl subset hull}

Recall that an ELA-subfield of $\B$ is an exponential subfield on which $\exp$ is surjective (i.e. every element has a logarithm) and which is field-theoretically algebraically closed.

\begin{proposition}\label{prop: strong ELA B=acl(B)}
    If $B \strong \B$ and $B$ is an ELA-subfield then $B = \acl_\B(B)$.
\end{proposition}
\begin{proof}
    Let $a_1\in \B \setminus B$. We will show that $a_1\notin \acl(B)$. By \cite[Proposition 3.2]{Kirby-pseudo} the (unique) type of an element which is exponentially transcendental over $B$ is not isolated, so in in particular not algebraic. Hence we may assume that $a_1$ is exponentially algebraic over $B$.

    Let $a_1,\ldots,a_n$ be a good basis for $\hull{Ba_1}$ over $B$ (replacing $a_1$ by some $\frac{a_1}{m}$ for an integer $m$ if necessary) and let $\abar:=(a_1,\ldots,a_n)$. Define $V:= \Loc(\abar, e^\abar/B)$. Since $B$ is an ELA-subfield, $V$ is free. Furthermore, $V$ is rotund, for $B \strong \B$. Since $a_1$ is exponentially algebraic over $B$ and $\abar$ is a basis for $\hull{Ba_1}$ over $B$, we have $\delta(\abar/B)=0$, and so $\dim V = n$.

    It is enough to show that $a_1\notin \acl(\bbar)$ for any finite tuple $\bbar$ from $B$. Extending $\bbar$, we may assume $V$ is defined over $\bbar$ and $\bbar \strong B$, so also $\bbar \strong \B$. 

    By Proposition~\ref{prop: isolating types over strong sets}, $\tp(\abar/\bbar)$ is isolated by the formula $\xbar \in V^{\dagger}_{\bbar}$. Suppose we have $N$ distinct realisations of this type, say $\cbar_1, \ldots, \cbar_N$. Then by Strong Exponential-Algebraic Closedness there is  $\cbar_{N+1}\in V^{\dagger}_{\bbar, \cbar_1, \ldots, \cbar_N} $.

    So inductively  we can find infinitely many realisations of $\tp(\abar/\bbar)$, with all components of the tuples $\Q$-linearly independent, so in particular distinct. Hence, $\tp(a_1/\bbar)$ is not algebraic, and so $a_1\notin \acl(B)$.
\end{proof}

{In order to deal with the case when $B$ is not an ELA-subfield of $\B$, we need to establish several auxiliary facts.}

\iffalse
\begin{proposition}\label{prop: EA acl=hull}
    If $B \strong \B$ and $B$ is an EA-subfield containing $\tau$ then $B = \acl(B)$.
\end{proposition}
\begin{proof}
    Pick an arbitrary $a\in \B \setminus B$. We will show that $a \notin \acl(B)$. By Proposition~\ref{prop: strong ELA B=acl(B)}, it suffices to consider $a \in \gen{B}^{\ELA}\setminus B$. \textcolor{brown}{Then (from the proof of \cite[Theorem~2.18]{Kirby-FPEF}, or just directly), there is a good basis $a_1,\ldots,a_n$ of $\hull{B a}$ over $B$ such that for each $i=1,\ldots,n$, exactly one of $a_i$ and $e^{a_i}$ is transcendental over $B,a_1,e^{a_1},\ldots,a_{i-1},e^{a_{i-1}}$.} \VA{I don't know why this is true.}

    Since $B$ is an EA-subfield, we must have $e^{a_1}\in B$ and $a_1$ transcendental over $B$. So, by Lemma~\ref{lem: e^a in span, a notin span}, $a_1\notin \acl(B)$.

    However, by Proposition~\ref{prop: acl contains hull}, $a_1\in \acl(Ba)$. If $a\in \acl(B)$ then $\acl(Ba)=\acl(B)$, which gives a contradiction. So $a\notin \acl(B)$ as required.
\end{proof}
\fi

{Let $X \subseteq \B$ be a $\Q$-linear subspace and $a \in \gen{X}^{\ELA}$. Then it follows from the construction of the ELA-closure (see \cite{Kirby-FPEF}) that there is $n \in \N$ and a sequence $a_1,\ldots,a_n \in \B$ such that for each $i=1,\ldots,n$, at least one of $a_i$ and $e^{a_i}$ is algebraic over $\Q(X,\exp(X),a_1,e^{a_1},\ldots,a_{i-1},e^{a_{i-1}})$, and $a \in \Span_\Q(X,a_1,\ldots,a_n)$. In fact without loss of generality, $a_n=a$.}

If we take $n$ to be least possible for $X$ and $a$ then it follows that $a_1,\ldots,a_n$ are $\Q$-linearly independent over $X$ and (unless $a \in X$) that $n\ge 1$. In this case we say that $a_1,\ldots,a_n$ is a witnessing sequence that $a\in  \gen{X}^{\ELA}$. If $X \strong \B$ then in addition we have that for each $i=1,\ldots,n$, exactly one of $a_i$ and $e^{a_i}$ is algebraic over $\Q(X,\exp(X),a_1,e^{a_1},\ldots,a_{i-1},e^{a_{i-1}})$. It also follows that $\Span_{\Q}(X,a_1,\ldots,a_n)\lhd \B$ (by induction on $n$).

\begin{lemma}\label{hull in ELA-closure}
{Suppose $X \strong \B$ and that $a_1,\ldots,a_n$ is a witnessing sequence that $a \in \gen{X}^{\ELA}$. Then $a_1,\ldots,a_n$ is also a basis for $\hull{Xa}$ over $X$.}
\end{lemma}

\begin{proof}
Suppose not, and take a counterexample with minimal possible value of $n$. Then $H \leteq \hull{Xa}$ is a proper $\Q$-linear subspace of $\Span_\Q(X,a_1,\ldots,a_n)$, for the latter is strong in $\B$ and contains $X\cup \{ a \}$. Then $a_1 \notin H$, since otherwise we can replace $X$ by $\Span_\Q(X a_1)$ to reduce $n$.

{We have $a_n \in H$, so we can take $j$ least possible such that $a_j \in H$ (even allowing the sequence of $a_i$'s to change, but fixing $X$ and $a$). So $a_1,\ldots,a_{j-1}$ are $\Q$-linearly independent over $H$, but $H \strong \B$ and hence
\begin{align*}
    \td(a_1,\ldots,a_{j-1},e^{a_1},\ldots,e^{a_{j-1}}/H,\exp(H)) & = j-1\\
    & = \td(a_1,\ldots,a_{j-1},e^{a_1},\ldots,e^{a_{j-1}}/X,\exp(X)).
\end{align*}
So, considering forking independence in the field we have
\[a_1,\ldots,a_{j-1},e^{a_1},\ldots,e^{a_{j-1}} \ind_{X,\exp(X)}^{\mathrm{ACF}} H, \exp (H)\]
and in particular, as $a_j \in H$, by monotonicity and symmetry we have 
\[a_j,e^{a_j} \ind_{X,\exp(X)}^{\mathrm{ACF}} a_1,\ldots,a_{j-1},e^{a_1},\ldots,e^{a_{j-1}}.\]}

{Hence $\td(a_j,e^{a_j}/X,\exp(X)) = 1$ and so $\delta(a_j/X) = 0$. Therefore, we can reorder the witnessing sequence by moving $a_j$ to the start. But now we have $a_1 \in H$ which gives the same contradiction as above to the minimality of $n$.}
\end{proof}

\begin{lemma}\label{lem: e^a in span, a notin span}
    If $\bbar \strong \B$ with $\tau \in \bbar$ and $e^a\in \Span_{\Q}(\bbar)$ but $a\notin \Span_{\Q}(\bbar)$ then $a \notin \acl(\bbar)$.
\end{lemma}
\begin{proof}
    We have \[ 0 \leq \delta(a/\bbar) = \td(a,e^a/ \bbar, e^{\bbar}) - \ldim_{\Q}(a/\bbar) = \td(a/\bbar, e^{\bbar})-1. \]
    So $a$ is transcendental over $\bbar, e^{\bbar}$. Without loss of generality assume that $e^a=b_1$. Hence the variety $V:=\Loc(a, e^a / \bbar, e^{\bbar})$ is defined by a single equation $y=b_1$. So by Proposition \ref{prop: isolating types over strong sets}, $\tp(a/\bbar)$ is isolated by the formula
    \[ e^x=b_1 \wedge x \notin \Span_{\Q}(\bbar). \]
    The realisations of this formula are $a+m\tau$ for $m\in \Z$, hence the type is not algebraic.
\end{proof}

\begin{proposition}\label{prop: EA acl=hull}
    If $B \strong \B$ and $B$ is an EA-subfield containing $\tau$ then $B = \acl(B)$.
\end{proposition}
\begin{proof}
    Pick an arbitrary $a\in \B \setminus B$. We will show that $a \notin \acl(B)$. By Proposition~\ref{prop: strong ELA B=acl(B)}, it suffices to consider $a \in \gen{B}^{\ELA}\setminus B$. Take a witnessing sequence $a_1,\ldots,a_n=a$ as above.
{Then by Lemma~\ref{hull in ELA-closure} we have $\hull{Ba} = \Span_\Q(B,a_1,\ldots,a_n)$.}

Since $B$ is an EA-subfield, we must have $e^{a_1}\in B$ and $a_1$ transcendental over $B$. So, by Lemma~\ref{lem: e^a in span, a notin span}, $a_1\notin \acl(B)$.

    However, by Proposition~\ref{prop: acl contains hull}, $a_1\in \acl(Ba)$. If $a\in \acl(B)$ then $\acl(Ba)=\acl(B)$, which gives a contradiction. So $a\notin \acl(B)$ as required.
\end{proof}

\begin{proof}[Proof of Theorem \ref{thm:main-thm-acl=hull-EA}]
Let $B\seq \B$. By Proposition \ref{prop: acl contains hull}, we have $\hull{B}^{\EA} \seq \acl(B)$, hence $\acl(B) = \acl(\hull{B}^{\EA})$. Then by Proposition \ref{prop: EA acl=hull}, $\acl(\hull{B}^{\EA}) = \hull{B}^{\EA}$.    
\end{proof}

\section{Discussion of closure operators}\label{sec: final comments}

\subsection{Closure operators on $\B$}

Having shown that $\acl(B) = \hull{B}^{\EA}$ for $B \subseteq \B$, an obvious question is whether $\acl$ generally admits a simpler description. For example, if $a = \log \log 2$ (any determination of this multivalued function) then it is easy to see that $\hull{a} = \Span_{\Q}(a,e^a)$, so $\hull{a} \subseteq \gen{a}^E$. Might it generally be the case that $\hull{B} \subseteq \gen{B}^{\EA}$?

In fact it is not generally true. For example, 
let $b = \exp(a^2)$, where $\exp(\exp(a))=2$ as above. Then $\hull{b} = \Span_{\Q}(b,a^2,a,e^a)$, but $a, a^2, e^a \notin \gen{b}^{\EA}$, although we do have $\hull{b} \subseteq \gen{b}^{\ELA}$.

\newcommand{\dsubseteq}{\rotatebox{90}{\ensuremath{\supseteq}}}
In fact, given a subset $B \subs \B$, we have various closures of $B$ forming the following diagram:
\[\begin{diagram}[height=20pt,width=30pt]
B & \subs & \gen{B}^\mathrm{E} &\subs &\gen{B}^{\EA} &\subs &\gen{B}^{\ELA} \\
\dsubseteq & & \dsubseteq & & \dsubseteq & & \dsubseteq \\
\hull{B} & \subs & \hull{B}^{\mathrm{E}} & \subs & \hull{B}^{\EA} & \subs & \hull{B}^{\ELA} & \subs & \ecl(B)\\
&&&& \rotatebox{90}{=}\\
&&&& \acl(B)
\end{diagram}\]

One can make any of the inclusions strict or non-strict by choosing a suitable $B$. Indeed, it is evident that the inclusions on the two horizontal rows are in general strict. The example above shows that the left three vertical inclusions can be made strict as well. 
The proof of Proposition~7.5 of \cite{Kirby-FPEF} gives an example of a point $a$ where $\hull{a} \nsubseteq \gen{a}^\ELA$, which then implies that the rightmost vertical inclusion can be made strict too. This example is somewhat technical, involving a point $(\abar,e^\abar)$ on the intersection of three generic hyperplanes in $\B^6$, so we do not give the details here.

Model theorists are particularly interested in those closure operators which are pregeometries, but of those in the diagram, it is easy to see that none but $\ecl$ satisfies the exchange property. The fact that $\ecl$ is a pregeometry on any exponential field is \cite[Theorem~1.1]{Kirby-Schanuel}.

Given the description of $\acl$ in $\B$, it is natural to ask for a characterisation of the definable closure $\dcl$.
It is immediate that $\gen{B}^\mathrm{E} \subs \dcl(B) \subs \acl(B) = \hull{B}^\EA$ but these bounds are not particularly good.

While neither $i$ nor $\tau$ is in $\dcl(\emptyset)$, being indistinguishable from $-i$ and $-\tau$ respectively, we can define the sine function $x \mapsto \sin(x) = \frac{e^{ix} - e^{-ix}}{2i}$ without the parameter $i$, due to the symmetry with $-i$. 
So $\dcl(B)$ is closed under the sine function and when $i \notin B$ this may take us outside $\gen{B}^E$. 
Furthermore, we saw in the introduction that in the complex exponential field $\{\pi,-\pi\}$ is definable, but we have $\sin(\pi/2) = 1$ and $\sin(-\pi/2) = -1$, so we can use this property to see that $\pi \in \dcl(\emptyset)$ in $\Cexp$. We can then use this to define $\pi$ in $\B$. Thus, $\pi \in \dcl(\emptyset)$ in $\B$ too.

Observe that $\hull{\pi} = \Span_{\Q}(\pi,\tau)$ and $\tau \notin \dcl(\pi)$, so we do not generally have $\hull{B} \subseteq \dcl(B)$.

Apart from sine, there are many more definable functions such as $x \mapsto e^{e^{\sqrt{x}}}+ e^{e^{-\sqrt{x}}}$, suggesting there is no simple description of $\dcl$ in $\B$.

\subsection{Non-standard kernel}

We have described $\acl$ on $\B$ itself, so an obvious question to ask is ``what about other models of its first-order theory?'' If $M$ is elementarily equivalent to $\B$ and also has standard kernel, then from \S2 of \cite{Kirby-pseudo} it follows that $M$ satisfies all the axioms we gave for $\B$ except perhaps the countable closure property, and the same proof works to show that $\acl(B) = \hull{B}^{\EA}$.
{On the other hand if $M$ has non-standard kernel then the Schanuel property may fail and then our methods would break down. If the multiplicative Zilber-Pink conjecture is true, then one can work with the Schanuel property ``over the kernel'', and combining the methods of this paper with those of \cite{Kirby-Zilber-exp} may answer the question.} Without assuming that conjecture the question seems much more difficult.

\subsection{Bounded closures}

One important role that $\acl$ plays in the model theory of complete first-order theories is to characterise those types which cannot have unboundedly many realisations in elementary extensions. We also call this the  \emph{bounded closure} $\bcl$. The natural axiomatisation of $\B$ we gave is in the logic $\Loo(Q)$, and without the countable closure property we get an axiomatisation in the logic $\Loo$.  In these logics, the compactness theorem does not hold, and $\acl$ and $\bcl$ can differ. It is known that for the $\Loo$-theory of $\B$, the bounded closure $\bcl(B) = \hull{B}^{\ELA}$, and for the $\Loo(Q)$-theory we have $\bcl(B) = \ecl(B)$.

\subsection{The complex case}

It is an open question whether or not $\R$ is definable in $\Cexp$. If it is, then $\Cexp$ is bi-interpretable with second order arithmetic, which has definable Skolem functions, and $\acl(B) = \dcl(B \cup \{i\})$ and is (bi-interpretable with) the smallest elementary submodel of second-order arithmetic containing $B$. This is the best upper bound we can get for $\acl(B)$ without something like the homogeneity obtained from Strong Exponential-Algebraic Closedness.

There is the obvious lower bound $\acl(B) \supseteq \gen{B \cup \{\pi\}}^\EA$ as discussed in the introduction. However, without the Schanuel Property, one does not have any control over $\hull{\bbar}$ for a finite tuple $\bbar$, and so the arguments of this paper do not show that $\hull{\bbar} \subseteq \acl(\bbar)$. Things are better if we work over $\hull{\emptyset}$, and we expect that the methods of this paper would show that if $B \supseteq \hull{\emptyset}$ then $\acl(B) \supseteq \hull{B}^{\EA}$. If Schanuel's conjecture is true then $\hull{\emptyset} = 2\pi i \Q$ and we recover this lower bound for all $B \subseteq \C$.

\subsection*{Acknowledgement.} {We thank the referee for useful remarks that helped us improve the presentation.}

\bibliographystyle{alpha}
\bibliography{ref}

\end{document}

%% file: macros.tex
 \newcommand{\VA}[1]{ \textbf{\textsl{\color{red}[VA: #1]}}}    % for notes to be removed

\DeclareMathOperator{\pr}{pr}  % Projection map

\DeclareMathOperator{\seq}{\subseteq}  % subset

\newcommand{\abar}{{\ensuremath{\bar{a}}}}
\newcommand{\bebar}{{\ensuremath{\bar{\beta}}}}
\newcommand{\cbar}{{\ensuremath{\bar{c}}}}
\newcommand{\bbar}{{\ensuremath{\bar{b}}}}

\DeclareMathOperator{\tp}{tp}  %type
\DeclareMathOperator{\rk}{rk}

\DeclareMathOperator{\acl}{acl}   % algebraic closure operator
\DeclareMathOperator{\dcl}{dcl}   % definable closure operator

\DeclareMathOperator{\td}{td}  %transcendence degree
\DeclareMathOperator{\loc}{Loc}   % (algebraic) locus

\newcommand{\alg}{\ensuremath{\mathrm{alg}}} %as index for alg closure
\DeclareMathOperator{\ldim}{ldim}  %linear dimension
\DeclareMathOperator{\ecl}{ecl} % exponential algebraic closure

\newcommand{\Id}{\ensuremath{\mathrm{Id}}}   % Identity map
\newcommand{\id}{\Id}

\newcommand{\N}{\ensuremath{\mathbb{N}}}
\newcommand{\Z}{\ensuremath{\mathbb{Z}}}
\newcommand{\Q}{\ensuremath{\mathbb{Q}}}
\newcommand{\R}{\ensuremath{\mathbb{R}}}
\newcommand{\C}{\ensuremath{\mathbb{C}}}
\newcommand{\Rexp}{\ensuremath{\mathbb{R}_{\mathrm{exp}}}}
\newcommand{\Cexp}{\ensuremath{\mathbb{C}_{\mathrm{exp}}}}

\newcommand{\Loo}{\ensuremath{L_{\omega_1,\omega}}}

\newcommand{\ga}{\ensuremath{\mathbb{G}_\mathrm{a}}}   %additive group of a field
\newcommand{\gm}{\ensuremath{\mathbb{G}_\mathrm{m}}}  %mult group of a field
\newcommand{\GL}{\ensuremath{\mathrm{GL}}}  %general linear group

\renewcommand{\phi}{\varphi}
\renewcommand{\le}{\ensuremath{\leqslant}}
\renewcommand{\ge}{\ensuremath{\geqslant}}

\newcommand{\class}[2]{\ensuremath{\left\{ #1 \,\left|\, #2 \right.\right\}}}

\newcommand{\subs}{\subseteq} % diagrams package uses subset for hook
\newcommand{\strong}{\ensuremath{\lhd}} % strong embedding - better
\newcommand{\nstrong}{\ensuremath{\not\kern-4pt\lhd\;}} % nonstrong embedding

\newcommand{\gen}[1]{\ensuremath{\left\langle #1 \right\rangle}} 
\newcommand{\hull}[1]{\ensuremath{\lceil #1\rceil}}

\newcommand{\cross}{\ensuremath{\times}}

\newbox\noforkbox \newdimen\forklinewidth
\forklinewidth=0.3pt \setbox0\hbox{$\textstyle\smile$}
\setbox1\hbox to \wd0{\hfil\vrule width \forklinewidth depth-2pt
  height 10pt \hfil}
\wd1=0 cm \setbox\noforkbox\hbox{\lower 2pt\box1\lower
2pt\box0\relax}
\def\unionstick{\mathop{\copy\noforkbox}\limits}

\def\nonfork_#1{\unionstick_{\textstyle #1}}

\setbox0\hbox{$\textstyle\smile$} \setbox1\hbox to \wd0{\hfil{\sl
/\/}\hfil} \setbox2\hbox to \wd0{\hfil\vrule height 10pt depth
-2pt width
                \forklinewidth\hfil}
\newbox\doesforkbox
\setbox\doesforkbox\hbox{\lower 2pt\box1 \lower
2pt\box2\lower2pt\box0\relax}
\def\nunionstick{\mathop{\copy\doesforkbox}\limits}

\def\fork_#1{\nunionstick_{\textstyle #1}}

\newcommand{\leteq}{\mathrel{\mathop:}=}